\documentclass[]{article}
\usepackage{indentfirst}
\usepackage{amsmath}
\usepackage{authblk}
\usepackage{geometry}
\usepackage{accents}
\usepackage{statex}
\usepackage[francais,english]{babel}
\usepackage{tikz}
\usetikzlibrary{arrows,automata,positioning}
\usepackage{authblk}

\newtheorem{theorem}{Theorem}
\newtheorem{proposition}{Proposition}

\newtheorem{corollary}{Corollary}
\newtheorem{remark}{Remark}

\newenvironment{proof}[1][Proof]{\begin{trivlist}
		\item[\hskip \labelsep {\bfseries #1}]}{\end{trivlist}}

\author[1]{Shuo LI}
\affil[1]{shuo.li.ismin@gmail.com}

\title{Zeckendorf expansion, Dirichlet series and infinite series involving the infinite Fibonacci word}

\date {}

\begin{document}
	
\maketitle

\begin{abstract}
Let $\beta=\frac{1+\sqrt{5}}{2}$, $(a_n)_{n \in \mathbb{N}^+}$ be a non-uniform morphic sequence involving the infinite Fibonacci word and $(\delta(n))_{n \in \mathbb{N}^+}$ be a positive sequence such that for all positive integers $n$, $\delta(n)=\frac{1}{\sqrt{5}}\sum_{j \geq 0}\epsilon_j\beta^{j+2}$ if the unique Zeckendorf expansion of $n$ is $n=\sum_{j \geq 0}\epsilon_jF_{j+2}$ with Fibonacci numbers $F_0,F_1,F_2...$. We define and study some Dirichlet series in the form of $\sum_{n\geq 1}\frac{a_n}{(\delta(n))^s}$ and relations between them. Moreover, we compute the values of some infinite series involving the infinite Fibonacci word.
\end{abstract}

\section{Introduction and main results}

Several infinite products and series involving sum-of-digits functions as well as block-counting functions were extensively studied in \cite{Allouche_1985}\cite{Allouche_1989}\cite{Allouche2007}\cite{Allouche_2018}\cite{Hu_2016}. The major ideal in these articles is that sum-of-digits functions and block-counting functions are all closely related to the notion of the automatic sequence, which can be defined as the images of the fixed points of uniform morphisms under some codings (more background is given in \cite{Allouche_2003}). However, we may expect to compute infinite products or series alone with functions in a more general class, for example, the morphic words. Herein, we offer a formal definition of the morphic words: let $\sum$ and $\Delta$ denote two non-empty sets of symbols and let $\sum^*$ and $\Delta^*$ respectively denote the free monoids on $\sum$ and $\Delta$. A sequence $(a_n)_{n \in \mathbb{N}^+} \in \Delta^*$ is called morphic if there exists a sequence $(e_n)_{n \in \mathbb{N}^+} \in \sum^*$, a morphism $f: \sum^* \to \sum^*$ and a morphism called coding $\rho: \sum \to \Delta$ such that $(a_n)_{n \in \mathbb{N}^+}=\rho((e_n)_{n \in \mathbb{N}^+})$ under the condition that the sequence $(e_n)_{n \in \mathbb{N}^+}$ is a fixed point of the morphism $f$. Particularly, if the morphism $f$ is uniform, that is to say the length of $f(a)$ is constant for all elements $a \in \sum$, the sequence $(a_n)_{n \in \mathbb{N}^+}$ is called automatic. A typical example of automatic sequences is the Thue-Morse sequence, which is the fixed point of the morphism $0 \to 0,1$ and $1 \to 1,0$. In this article, we focus on non-uniform morphic sequences involving the infinite Fibonacci word, which is the fixed point of the non-uniform morphism $0\to0$ and $1 \to 0,1$. The main ideal of this article arises from the fact that the infinite Fibonacci word can be deduced by counting the number of $0$s at the end of the Zeckendrof expansion of each integer. 

Let the Fibonacci numbers be defined by $F_0 = 0, F_1 = 1$, and $F_n = F_{n-1}+F_{n-2}$ for $ n \geq 2$. From Zeckendorf’s theorem \cite{zeck}, every positive integer $n$ can be uniquely written as $n=\sum_{j \geq 0}\epsilon_jF_{j+2}$ with $\epsilon_j \in \left\{0,1\right\}$ under the condition that $\epsilon_j\epsilon_{j+1}=0$ for all $j \geq 0$. This expansion is called the Zeckendorf expansion. Using this notion, we can define some other sequences which present useful properties in analysis as well as in combinatorics. On the one hand, if we let $\beta$ denote $\frac{1+\sqrt{5}}{2}$, then $F_n=\frac{1}{\sqrt{5}}(\beta^n-(-\beta)^{-n})$ for all $n \geq 0$. Thus, it is natural to define and study the following two sequences $(\delta(n))_{n \in \mathbb{N}}$ and $(\delta'(n))_{n \in \mathbb{N}}$: for all integers $n \geq 0$, if $n=\sum_{j \geq 0}\epsilon_jF_{j+2}$ is the Zeckendorf expansion of $n$, then
$$\delta(n)=\frac{1}{\sqrt{5}}\sum_{j \geq 0}\epsilon_j\beta^{j+2} \; \text{and} \; \delta'(n)=\frac{1}{\sqrt{5}}\sum_{j \geq 0}\epsilon_j(-\beta)^{-j-2}.$$
In Section two, we will study the arithmetic properties of the sequences $(\delta(n))_{n \in \mathbb{N^+}}$ and $(\delta'(n))_{n \in \mathbb{N^+}}$, as well as some properties of the Dirichlet series $F(s)=\sum_{n\geq 1}\frac{1}{(\delta(n))^s}$. On the other hand, some morphic sequences can also be defined by using the Zeckendorf expansion, namely, the Fibonacci-automatic sequences defined and studied in \cite{Mousavi2016}\cite{Du2016}\cite{Du2017}. Among these sequences, a typical example is the infinite Fibonacci sequence $(f(n))_{n \in \mathbb{N^+}}$. In Section three, we study some combinatorial properties of the infinite Fibonacci sequence and other morphic sequences. In Section four, we present a combinatorial proof of the meromorphic continuation of $F$ on the whole complex plane; the main theorem is announced as follows:
\begin{theorem}
The Dirichlet series $F$ converges absolutely on the set $\left\{s| \Re(s)>1\right\}$, and has a meromorphic continuation on the whole complex plane. Moreover, its poles are located on the set of zeros of the function $s \to 1-2\beta^{-s}+\beta^{-3s}$.
\end{theorem}
Moreover, we extend this result to other Dirichlet series with morphic coefficients. Similar results can be found in \cite{Sourmelidis2019}.
Finally, in Section five, we use the method introduced in \cite{Allouche_1985} to compute some infinite series:
\begin{theorem}
Letting $(f(n))_{n \in \mathbb{N}^+}$, $(d(n))_{n \in \mathbb{N}^+}$, $(r(n))_{n \in \mathbb{N}^+}$, $(s(n))_{n \in \mathbb{N}^+}$ and $(t(n))_{n \in \mathbb{N}^+}$ be integer sequences defined in the following sections and letting $\beta=\frac{\sqrt{5}+1}{2}$, then we have
$$\sum_{n \geq 2}r(n)(\frac{\sqrt{5}}{(n-1)\sqrt{5}-\{\beta (n-1)\}+1-f(n-1)}-\frac{\sqrt{5}}{n\sqrt{5}-\{\beta n\}+1-f(n)})=\frac{\beta-1}{\beta^2}\ln(\beta),$$
$$\sum_{\substack{n \geq 1\\d(n)=1}}(\frac{\sqrt{5}}{(n-1)\sqrt{5}-\{\beta (n-1)\}+1-f(n-1)}-\frac{\sqrt{5}}{n\sqrt{5}-\{\beta n\}+1-f(n)})=\beta^{-1}\ln(\beta),$$
$$\sum_{\substack{n \geq 1\\d(n)=2}}(\frac{\sqrt{5}}{(n-1)\sqrt{5}-\{\beta (n-1)\}+1-f(n-1)}-\frac{\sqrt{5}}{n\sqrt{5}-\{\beta n\}+1-f(n)})=\beta^{-2}\ln(\beta),$$
$$\sum_{n \geq 2}t(n)(\frac{\sqrt{5}}{(n-1)\sqrt{5}-\{\beta (n-1)\}+1-f(n-1)}-\frac{\sqrt{5}}{n\sqrt{5}-\{\beta n\}+1-f(n)})=\beta^2-\frac{\beta-1}{\beta^2}\ln(\beta),$$
$$\sum_{n \geq 2}\frac{s(n)\sqrt{5}}{n\sqrt{5}-\{\beta n\}+1-f(n)}=(\frac{3}{2}\beta^{-4}-\beta^{-2})\ln(\beta),$$
where $\{x\}$ is the fractional part of $x$.
\end{theorem}

\section{Arithmetic properties of $(\delta(n))_{n \in \mathbb{N^+}}$}

  \begin{proposition}
    Let $e(n)$ denote the last bit in the Zeckendorf expansion of $n$; that is to say, for a given integer $n$, $e(n)$ is the coefficient $\epsilon_0$ in the expansion $n=\sum_{j \geq 0}\epsilon_jF_{j+2}$, and we then have the following equations:
    %Editor: Please be consistent with the use of punctuation when introducing equations (i.e., the use of a colon or not).

$$\delta(n+1)-\delta(n)=\begin{cases}
\frac{1}{\sqrt{5}}\beta^2 \; \; \;\text{if} \; e(n)=0,\\
\frac{1}{\sqrt{5}}\beta \;\;\;\;\text{if} \; e(n)=1.
\end{cases}$$
Consequently, the sequence $(\delta(n))_{n \in \mathbb{N}}$ is an unbounded increasing sequence.
\end{proposition}

\begin{proof}
For a given integer $n$, if $e(n)=1$, from the convention that there does not exist the factor $\overline{1,1}$ in the Zeckendorf expansion of $n$, we can suppose that this expansion is in the form $\overline{\epsilon_k,\epsilon_{k-1},\epsilon_{k-2},...,\epsilon_s, 0,\underbrace{ 0,1 }_{i\; \text{times}}}$, where $\epsilon_i$ are either $0$ or $1$. From the recurrent relations between the Fibonacci numbers and the uniqueness of the Zeckendrof expansion, the expansion of $n+1$ is in the form $\overline{\epsilon_k,\epsilon_{k-1},\epsilon_{k-2},...,\epsilon_s, 0,1,\underbrace{ 0 }_{2i-1\; \text{times}}}$. Therefore,
$$\delta(n+1)-\delta(n)=\frac{1}{\sqrt{5}}(\beta^{2i+1}-\sum_{k=1}^{i}\beta^{2k}).$$
Using the equation $\beta^2-\beta=1$ recurrently, we have $\delta(n+1)-\delta(n)=\frac{1}{\sqrt{5}}\beta$.

In the same way, if $e(n)=0$ and the Zeckendorf expansion of $n$ ends up with a suffix $\overline{1,0}$, then we can suppose that this expansion is in the form $\overline{\epsilon_k,\epsilon_{k-1},\epsilon_{k-2},...,\epsilon_s, 0,\underbrace{ 0,1 }_{i\; \text{times}},0}$. Thus, the expansion of $n+1$ is in the form $\overline{\epsilon_k,\epsilon_{k-1},\epsilon_{k-2},...,\epsilon_s, 0,1,\underbrace{ 0 }_{2i\; \text{times}}}$. Therefore,
$$\delta(n+1)-\delta(n)=\frac{1}{\sqrt{5}}(\beta^{2i+2}-\sum_{k=1}^{i}\beta^{2k+1})=\frac{1}{\sqrt{5}}\beta^2.$$

In the last case, if $e(n)=0$ but the Zeckendorf expansion of $n$ ends up with a suffix $\overline{0,0}$, then we can suppose that this expansion is in the form $\overline{\epsilon_k,\epsilon_{k-1},\epsilon_{k-2},...,\epsilon_s, 0,0}$. Thus, the expansion of $n+1$ is in the form $\overline{\epsilon_k,\epsilon_{k-1},\epsilon_{k-2},...,\epsilon_s, 0,1}$, so that 
$$\delta(n+1)-\delta(n)=\frac{1}{\sqrt{5}}\beta^2.$$
\end{proof}

\begin{proposition}
Let us define the sequence $(d(n))_{n \in \mathbb{N^+}}$ in the following way:

$$d(n)=\begin{cases}
0 \; \text{if the Zeckendorf expansion of $n$ admits the string $\overline{1}$ as a suffix},\\
1 \; \text{if the Zeckendorf expansion of $n$ admits a string of type $\overline{1,\underbrace{ 0 }_{2i+1; \text{times}}}$ as a suffix},\\
2 \; \text{if the Zeckendorf expansion of $n$ admits a string of type $\overline{1,\underbrace{ 0 }_{2i+2; \text{times}}}$ as a suffix},
\end{cases}$$
where $i$ is a positive integer. We then have $|\delta'(n)|<\frac{1}{\beta\sqrt{5}}$  for all $n>0$.
Moreover, 
$$
\begin{cases}
\delta'(n)>0 \; \text{if}\;d(n)= 0 \; \text{or}\; d(n)=2,\\
\delta'(n)<0 \; \text{if} \; d(n)=1.
\end{cases}
$$

\end{proposition}

\begin{proof}
It is clear that $|-\beta^{-1}|<1$. For a given integer $n$, letting $n=\sum_{j \geq 0}\epsilon_jF_{j+2}$ with $\epsilon_j \in \left\{0,1\right\}$ be the Zeckendorf expansion of $n$, then

\begin{equation}
  \begin{aligned}
        |\delta'(n)|&=|\frac{1}{\sqrt{5}}\sum_{i\geq 0}\epsilon_i(-\beta)^{-i-2}|\\
        &\leq \frac{1}{\sqrt{5}}\sum_{i\geq 0}\epsilon_i(\beta^{-i-2})\\
        &\leq \frac{1}{\sqrt{5}}\sum_{i\geq 0}(\beta^{-2i-2})\\
        &\leq \frac{1}{\sqrt{5}}\beta^{-2}\frac{1}{1-\beta^{-2}}.\\
    \end{aligned}
\end{equation}
 Substituting the equation $\beta^2-\beta-1=0$ into Equation 1, we have $|\delta'(n)| \leq  \frac{1}{\sqrt{5}}\beta^{-1}\frac{\beta^{-1}}{1-\beta^{-2}}= \frac{1}{\sqrt{5}}\beta^{-1}$.

For the second part of the proposition, let $n$ be an integer and let $i$ be the smallest index such that $\epsilon_i=1$ in the Zeckendorf expansion of $n$. It is easy to verify that $(-\beta)^{-i-2}>0$ if $d(n)=0$ or $2$, and $(-\beta)^{-i-2}<0$ if $d(n)=1$. It follows from Equation 1 that $|\frac{1}{\sqrt{5}}\sum_{k\geq i+1}\epsilon_k(-\beta)^{-k-2}|\leq\beta^{-i-3}$, so the sign of $\delta'(n)$ is the same as that of $(-\beta)^{-i-2}$.
\end{proof}

\begin{proposition}
The function $F(s)=\sum_{n \geq 1} \frac{1}{(\delta(n))^s}$ is a Dirichlet series which converges absolutely on $\left\{s|\Re(s)>1\right\}$ and has a meromorphic continuation on $\left\{s|\Re(s)>0\right\}$. Moreover, $F(s)$ has a single pole at $s=1$ with residue $1$. 
\end{proposition}

\begin{proof}
From Propositions 1 and 2, the sequence $(\delta(n))_{n \in \mathbb{N}}$ is increasing and $n-1<\delta(n) < n+1$ for all positive integers $n$. Thus, $F(s)$  converges absolutely on $\left\{s|\Re(s)>1\right\}$. For the extension, let $\zeta$ be the Riemann zeta function,

\begin{equation}
  \begin{aligned}
        \zeta(s)-F(s)&=\sum_{n \geq 1}\frac{1}{n^s}-\frac{1}{\delta^s(n)}\\
        &=\sum_{n \geq 1}\frac{1}{(\delta(n)-\delta'(n))^s}-\frac{1}{\delta^s(n)}\\
        &=\sum_{n \geq 1}\frac{1}{\delta^s(n)}\frac{1}{(1-\frac{\delta'(n)}{\delta(n)})^s}-\frac{1}{\delta^s(n)}\\
        &=\sum_{n \geq 1}\frac{1}{\delta^s(n)}(\sum_{m\geq 0}\binom{-s}{m}(-\frac{\delta'(n)}{\delta(n)})^m-1)\\
        &=\sum_{m\geq 1}\binom{-s}{m}\sum_{n \geq 1}\frac{(-\delta'(n))^m}{\delta^{s+m}(n)}
    \end{aligned}
\end{equation}

For any given positive integer $m$, $|\sum_{n \geq 1}\frac{(-\delta'(n))^m}{\delta^{s+m}(n)}|\leq \sum_{n \geq 1}|\frac{(\delta'(n))^m}{\delta^{s+m}(n)}| \leq \sum_{n \geq 1}\frac{\beta^{-m}}{\delta^{\Re(s)+m}(n)}$. Since the term $\sum_{n \geq 1}\frac{1}{\delta^{\Re(s)+m}(n)}$ is bounded for large $m$, the righthand side of equation (2) converges for all $s$ such that $\Re(s)>0$. Consequently, the function $F(s)$ has a meromorphic continuation on $\left\{s|\Re(s)>0\right\}$ and has the same pole with the same residue as the Riemann zeta function on $s=1$.
\end{proof}

\section{Fibonacci sequence and its first differences sequence}

Let $(f(n))_{n \in \mathbb{N^+}}$ be the Fibonacci sequence defined as the fixed point of the morphism $1 \to 0,1$ and $0 \to 0$ and let us recall the sequence $(d(n))_{n \in \mathbb{N^+}}$ defined in Proposition 2:

$$d(n)=\begin{cases}
0 \; \text{if the Zeckendorf expansion of $n$ admits the string $\overline{1}$ as a suffix}\\
1 \; \text{if the Zeckendorf expansion of $n$ admits a string of type $\overline{1,\underbrace{ 0 }_{2i+1; \text{times}}}$ as a suffix}\\
2 \; \text{if the Zeckendorf expansion of $n$ admits a string of type $\overline{1,\underbrace{ 0 }_{2i+2; \text{times}}}$ as a suffix},
\end{cases}$$with some positive integer $i$. In this section, we will show the relations between $(d(n))_{n \in \mathbb{N^+}}$, $(f(n))_{n \in \mathbb{N^+}}$ and the first differences sequence of $(f(n))_{n \in \mathbb{N^+}}$.

 \begin{proposition}
The sequence $(d(n))_{n \in \mathbb{N^+}}$ satisfies the following properties:

1, if $d(n)=0$, then $d(n+1)=1$;

2, if $d(n)=1$, then $d(n+1)=2$ or $0$;

3, if $d(n)=2$, then $d(n+1)=0$.
\end{proposition}

\begin{proof}
If $d(n)=0$, then the Zeckendorf expansion of $n$ is in the form $\overline{\epsilon_k,\epsilon_{k-1},\epsilon_{k-2},...,\epsilon_s, 0,\underbrace{0,1 }_{i\; \text{times}}}$ for some $i\geq 1$. From the recurrent relations between the Fibonacci numbers and the uniqueness of the Zeckendrof expansion , the expansion of $n+1$ is in the form $\overline{\epsilon_k,\epsilon_{k-1},\epsilon_{k-2},...,\epsilon_s,0,1,\underbrace{ 0 }_{2i-1\; \text{times}}}$. Consequently, $d(n+1)=1$.

If $d(n)=2$, then the Zeckendorf expansion of $n$ is in the form $\overline{\epsilon_k,\epsilon_{k-1},\epsilon_{k-2},...,\epsilon_s, 0,0}$. Thus, the expansion of $n+1$ is in the form $\overline{\epsilon_k,\epsilon_{k-1},\epsilon_{k-2},...,\epsilon_s, 0,1}$. Therefore, $d(n+1)=0$.

If $d(n)=1$, from the definition, the Zeckendorf expansion of $n$ ends up with a suffix $\overline{1,\underbrace{ 0 }_{2i+1\; \text{times}}}$ for some $i \geq 0$. There are two cases:  if $i=0$, then the Zeckendorf expansion of $n$ is in the form $\overline{\epsilon_k,\epsilon_{k-1},\epsilon_{k-2},...,\epsilon_s, 0,\underbrace{ 0,1 }_{i\; \text{times}},0}$. Thus, the expansion of $n+1$ is in the form $\overline{\epsilon_k,\epsilon_{k-1},\epsilon_{k-2},...,\epsilon_s, 0,1,\underbrace{ 0 }_{2i\; \text{times}}}$. In this case, $d(n+1)=2$.
 If $i\geq1$, then the Zeckendorf expansion of $n$ is in the form $\overline{\epsilon_k,\epsilon_{k-1},\epsilon_{k-2},...,\epsilon_s, 0,0}$. Thus, the expansion of $n+1$ is in the form $\overline{\epsilon_k,\epsilon_{k-1},\epsilon_{k-2},...,\epsilon_s, 0,1}$. Therefore, $d(n+1)=0$.

\end{proof}

 \begin{proposition}
The sequence $(f(n))_{n \in \mathbb{N^+}}$ is the image of $(d(n))_{n \in \mathbb{N^+}}$ under the morphism $0 \to 0,\; 1\to 1$ and $2 \to 0$. Moreover, if we define $(h(n))_{n \in \mathbb{N^+}}$ as the first differences sequence of $(f(n))_{n \in \mathbb{N^+}}$, that is to say, $h(n)=f(n)-f(n+1)$, then $(h(n))_{n \in \mathbb{N^+}}$ is the image of $(d(n))_{n \in \mathbb{N^+}}$ under the morphism $0 \to -1,\; 1\to 1$ and $2 \to 0$.
\end{proposition}

\begin{proof}
It is mentioned in \cite{Mousavi2016} that $f(n)$ is the output of the following direct automaton
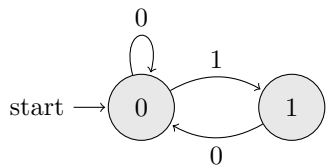
\begin{center}
\begin{tikzpicture}[shorten >=1pt,node distance=2cm,on grid,auto]

  \tikzstyle{every state}=[fill={rgb:black,1;white,10}]

  \node[state,initial]  	(s_0)                 {$0$};
  \node[state]                    (s_1) [right of=s_0]  {$1$};

  \path[->]
  (s_0) edge  [loop above]  node {0}  (s_0)
	edge   [bend left]            node {1}  (s_1)
  (s_1) edge  [bend left]  node {0}  (s_0);

\end{tikzpicture}
\end{center}
when the input is the Zeckendorf expansion of $n-1$. In other words, $d(n)=1$ if and only if the Zeckendorf expansion of $n-1$ ends up with a value of $1$ and $d(n)=0$ if and only if the Zeckendorf expansion of $n-1$ ends up with a value of $0$. On the other hand, from Proposition 4, if $d(n)=1$, then $d(n-1)=0$, so that the Zeckendorf expansion of $n-1$ ends up with a value of $1$, and thus, $f(n)=1$; similarly, if $d(n)=0$ or $2$, then $d(n-1)=1$ or $2$, so that the Zeckendorf expansion of $n-1$ ends up with a value of $0$, and thus, $f(n)=0$. The second part of this proposition is a direct consequence of the previous result.
\end{proof}

\begin{remark}
From the previous proposition, the sequence $(f(n))_{n \in \mathbb{N^+}}$ is the output of the direct automaton
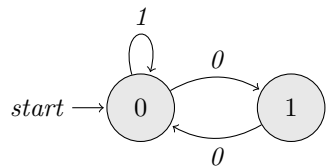
\begin{center}
\begin{tikzpicture}[shorten >=1pt,node distance=2cm,on grid,auto]

  \tikzstyle{every state}=[fill={rgb:black,1;white,10}]

  \node[state,initial]  	(s_0)                 {$0$};
  \node[state]                    (s_1) [right of=s_0]  {$1$};

  \path[->]
  (s_0) edge  [loop above]  node {1}  (s_0)
	edge   [bend left]            node {0}  (s_1)
  (s_1) edge  [bend left]  node {0}  (s_0);

\end{tikzpicture}
\end{center}
when the input is the Zeckendorf expansion of $n$. Moreover, from the descriptions of the sequences A00384, A001468, A014677 and A270788\cite{oeis}, the sequence $(d(n))_{n \in \mathbb{N^+}}$ is the image of the morphic sequence A270788 under the coding $1 \to 0, \; 2 \to 1,\; 3 \to 2$.
\end{remark}

 \begin{proposition}
Let us define two functions $\tau_0$, $\tau_1: \mathbb{N} \to \mathbb{N}$ in the following way: for every positive integer $n$, letting $n=\sum_{j \geq 0}\epsilon_jF_{j+2}$ be the Zeckendorf expansion, then 
$$\tau_0(n)=\sum_{j \geq 0}\epsilon_jF_{j+3};$$
$$\tau_1(n)=\sum_{j \geq 0}\epsilon_jF_{j+3}+F_2.$$
The following relations exist between the sets:
$$\left\{\tau_0(n)|n\in \mathbb{N}^+\right\} =\left\{n|n\in \mathbb{N}^+, d(n)=1 \; \text{or}\; 2\right\},$$
$$\left\{\tau_1(n)|n\in \mathbb{N}^+,  d(n)=1 \; \text{or}\; 2\right\}=\left\{n|n\in \mathbb{N}^+, n \geq 2, d(n)=0\right\},$$
$$\left\{\tau_1(n)|n\in \mathbb{N}^+,  d(n)=0\right\}=\left\{n|n\in \mathbb{N}^+, d(n)=2\right\}.$$
Consequently,
$$\left\{\tau_0(n)|n\in \mathbb{N}^+\right\} \cup \left\{\tau_1(n)|n\in \mathbb{N}^+\right\}=\left\{n|n\in \mathbb{N}^+, n \geq 2\right\},$$
$$\left\{\tau_0(n)|n\in \mathbb{N}^+\right\} \cap \left\{\tau_1(n)|n\in \mathbb{N}^+\right\}=\left\{n|n\in \mathbb{N}^+, d(n)=2\right\}.$$
\end{proposition}

\begin{proof}
If $d(n)=1$ or $2$, then the Zeckendorf expansion of $n$ is in the form $\overline{\epsilon_k,\epsilon_{k-1},\epsilon_{k-2},...,\epsilon_s, 0}$. Thus, the expansion of $\tau_0(n)$ is in the form $\overline{\epsilon_k,\epsilon_{k-1},\epsilon_{k-2},...,\epsilon_s, 0,0}$ and the expansion of $\tau_1(n)$ is in the form $\overline{\epsilon_k,\epsilon_{k-1},\epsilon_{k-2},...,\epsilon_s, 0,1}$. Therefore, $d(\tau_0(n))=1$ or $2$ and $d(\tau_1(n))=0$.

 If $d(n)=0$, then the Zeckendorf expansion of $n$ is in the form $\overline{\epsilon_k,\epsilon_{k-1},\epsilon_{k-2},...,\epsilon_s, 0,\underbrace{ 1,0 }_{i\; \text{times}},1}$ for some $i\geq 0$. Thus, the expansion of $\tau_0(n)$ is in the form $\overline{\epsilon_k,\epsilon_{k-1},\epsilon_{k-2},...,\epsilon_s, 0,\underbrace{ 1,0 }_{i+1\; \text{times}}}$ and the expansion of $\tau_1(n)$ is in the form $\overline{\epsilon_k,\epsilon_{k-1},\epsilon_{k-2},...,\epsilon_s, 1,\underbrace{ 0}_{2i+2\; \text{times}}}$. Therefore, $d(\tau_0(n))=1$ or $2$ and $d(\tau_1(n))=0$.
\end{proof}

\begin{corollary} 
$$\left\{\beta\delta(n)|n\in \mathbb{N}^+\right\} \cup \left\{\beta\delta(n)+\frac{1}{\sqrt{5}}\beta^2|n\in \mathbb{N}^+\right\}=\left\{\delta(n)|n\in \mathbb{N}^+, n \geq 1\right\} $$
$$\left\{\beta\delta(n)|n\in \mathbb{N}^+\right\} \cap \left\{\beta\delta(n)+\frac{1}{\sqrt{5}}\beta^2|n\in \mathbb{N}^+\right\}=\left\{\delta(n)|n\in \mathbb{N}^+, d(n)=2\right\} $$
\end{corollary}

\begin{proof}
It is directly from this fact that if $\delta(n)=\frac{1}{\sqrt{5}}\sum_{j \geq 0}\epsilon_j\beta^{j+2}$, then $\beta\delta(n)=\frac{1}{\sqrt{5}}\sum_{j \geq 0}\epsilon_j\beta^{j+3}$ and $\beta\delta(n)+\frac{1}{\sqrt{5}}\beta^2=\frac{1}{\sqrt{5}}(\sum_{j \geq 0}\epsilon_j\beta^{j+3}+\beta^2)$.
\end{proof}

\section{Dirichlet series involving the sequence $(d(n))_{n \in \mathbb{N}^+}$}

Let us recall the Dirichlet series $F(s)=\sum_{n \geq 1} \frac{1}{(\delta(n))^s}$, and define four other Dirichlet series:
$$G(s)=\sum_{\substack{n \geq 1\\d(n)=0}} \frac{1}{(\delta(n))^s};$$
$$H(s)=\sum_{n \geq 1} \frac{1}{(\beta\delta(n)+\frac{1}{\sqrt{5}}\beta^2)^s};$$ 
$$I(s)=\sum_{\substack{n \geq 1\\d(n)=1}} \frac{1}{(\delta(n))^s};$$
$$J(s)=\sum_{\substack{n \geq 1\\d(n)=2}} \frac{1}{(\delta(n))^s}.$$
In this section, we prove the meromorphic continuation of all of these series on the whole complex plane. First, let us prove Theorem 1.

\begin{proof}[Proof of Theorem 1]

From Corollary 1, if $\Re(s)>1$, then 

\begin{equation}
  \begin{aligned}
        F(s)&=\sum_{n \geq 2}\frac{1}{\delta^s(n)} + \frac{1}{\delta^s(1)}\\
        &=\sum_{n \geq 1}\frac{1}{(\beta\delta(n))^s}+\sum_{\substack{n \geq 1\\d(n)=0}}\frac{1}{(\delta(n))^s} + \frac{1}{\delta^s(1)}\\
        &=\beta^{-s}F(s)+G(s)+ \frac{1}{\delta^s(1)}
    \end{aligned}
\end{equation}

\begin{equation}
  \begin{aligned}
        G(s)&=\sum_{n \geq 1}\frac{1}{(\beta\delta(n)+\frac{1}{\sqrt{5}}\beta^2)^s}-\sum_{\substack{n \geq 1\\d(n)=2}}\frac{1}{(\delta(n))^s}\\
        &=H(s)-\sum_{m\geq 1}\sum_{\substack{n \geq 1\\d(n)=0}}\frac{1}{(\beta^{2m}\delta(n))^s}\\
        &=H(s)-\frac{\beta^{-2s}}{1-\beta^{-2s}}G(s)
    \end{aligned}
\end{equation}

Substituting (4) into (3), we have

\begin{equation}
       (1-\beta^{-s}) F(s)=(1-\beta^{-2s}) H(s) + \frac{1}{\delta^s(1)}
\end{equation}

Moreover, with the fact that $0 <\frac{1}{\sqrt{5}}\beta<1$, we have
\begin{equation}
  \begin{aligned}
      H(s)&=\sum_{n \geq 1} \frac{1}{(\beta\delta(n)+\frac{1}{\sqrt{5}}\beta^2)^s}\\
        &=\sum_{n \geq 1}\frac{1}{(\beta\delta(n))^s}\frac{1}{(1+\frac{\frac{1}{\sqrt{5}}\beta}{\delta(n)})^s}\\
        &=\sum_{n \geq 1}\frac{1}{(\beta\delta(n))^s}\sum_{m\geq 0}\binom{-s}{m}(\frac{\frac{1}{\sqrt{5}}\beta}{\delta(n)})^m\\
        &=\beta^{-s}\sum_{m\geq 0}(\frac{1}{\sqrt{5}}\beta)^m\binom{-s}{m}F(s+m)
    \end{aligned}
\end{equation}

From (5) and (6), we can deduce that

\begin{equation}
(1-2\beta^{-s}+\beta^{-3s})F(s)=(\beta^{-s}-\beta^{-3s})\sum_{m\geq 1}(\frac{1}{\sqrt{5}}\beta)^m\binom{-s}{m}F(s+m)+ \frac{1}{\delta^s(1)}
\end{equation}

For any given complex number $s$ such that $\Re(s)>1$, the sequence $(F(s+k))_{k \in \mathbb{N}}$ is bounded. Thus, the righthand side of equation (7) converges uniformly for $\Re(s)>0$. Hence, $F(s)$ has a meromorphic extension for $0<\Re(s) \leq 1$. Now, if $0<\Re(s) \leq 1$, the righthand side converges, with the exception of those $s$ for which $s$ is a zero of $1-2\beta^{-s}+\beta^{-3s}$. This yields a meromorphic extension of $F$ for $\Re(s)>-1$. Iterating this process shows that $F$ has a meromorphic extension to the whole complex plane. Moreover, the poles of $F$ are located on the set of zeros of the function $s \to 1-2\beta^{-s}+\beta^{-3s}$.
\end{proof}

\begin{corollary}
The Dirichlet series $G,H,I,J$ all admit meromorphic continuations on the whole complex plane and have simple poles at $s=1$. Moreover, their residues at $s=1$ are respectively $1-\beta^{-1}$,$\beta^{-1}$,$1-\beta^{-1}$ and $\beta^{-1}-\beta^{-2}$.
\end{corollary}

\begin{proof}
The meromorphic continuations of $G$ and $H$ are given respectively by (3) and (5), along with their residues.

To see the meromorphic continuation of $I$ and $J$, on the set $\left\{s| \Re(s)>1\right\}$, we have 

\begin{equation}
I(s)=\sum_{m\geq 0}\sum_{\substack{n \geq 1\\d(n)=0}}\frac{1}{(\beta^{2m+1}\delta(n))^s}=\frac{\beta^{-s}}{1-\beta^{-2s}}G(s);
\end{equation}
 
and 
\begin{equation}
J(s)=\sum_{m\geq 1}\sum_{\substack{n \geq 1\\d(n)=0}}\frac{1}{(\beta^{2m}\delta(n))^s}=\frac{\beta^{-2s}}{1-\beta^{-2s}}G(s).
\end{equation}

For the residues, we can use the fact that the residue of $F$ at $s=1$ is $1$.

\end{proof}

\begin{proposition}
Let $a$, $b$ be two real numbers such that $|b|\leq |a|$ and let $i=0,1,2$ or $3$. Letting $K_{a,b}^{(i)}(s)$ be the function
$$K_{a,b}^{(i)}(s)=\sum_{\substack{n \geq 1\\d(n)=i}} \frac{1}{(a\delta(n))^s}-\frac{1}{(a\delta(n)+b)^s}$$ for $i=0,1,2$ and letting $K_{a,b}^{(3)}(s)=\sum_{i=0}^2K_{a,b}^{(i)}(s)$, then the function $K_{a,b}^{(i)}(s)$ has a meromorphic continuation on the whole complex plane for any $i$. Moreover, it converges absolutely on $\left\{s| \Re(s)>1\right\}$, converges pointwisely on $\left\{s| \Re(s)> 0\right\}$ and $\lim_{s \to 0}K_{a,b}^{(0)}(s)=\frac{b}{a}(1-\beta^{-1})$, $\lim_{s \to 0}K_{a,b}^{(1)}(s)=\frac{b}{a}(1-\beta^{-1})$,$\lim_{s \to 0}K_{a,b}^{(2)}(s)=\frac{b}{a}(\beta^{-1}-\beta^{-2})$ and $\lim_{s \to 0}K_{a,b}^{(3)}(s)=\frac{b}{a}$.
\end{proposition}

\begin{proof}
From the hypothesis, we have $|\frac{b}{a\delta(n)}|<1$ for all $n$. Thus, for $i=0,1$ or $2$, on the set $\left\{s| \Re(s)>1\right\}$

\begin{equation}
  \begin{aligned}
 K_{a,b}^{(i)}(s)&=\sum_{\substack{n \geq 1\\d(n)=i}} (\frac{1}{(a\delta(n))^s}-\frac{1}{(a\delta(n)+b)^s})\\
        &=\sum_{\substack{n \geq 1\\d(n)=i}} \frac{1}{(a\delta(n))^s}(1-\frac{1}{(1+\frac{b}{a\delta(n)})^s})\\
        &=\sum_{\substack{n \geq 1\\d(n)=i}} \frac{1}{(a\delta(n))^s}(1-\sum_{m \geq 0}\binom{-s}{m}(\frac{b}{a\delta(n)})^m)\\
        &=-\sum_{\substack{n \geq 1\\d(n)=i}} \frac{1}{(a\delta(n))^s}\sum_{m \geq 1}\binom{-s}{m}(\frac{b}{a\delta(n)})^m\\
        &=-a^{-s}\sum_{m \geq 1}\binom{-s}{m}(\frac{b}{a})^m\sum_{\substack{n \geq 1\\d(n)=i}} \frac{1}{(\delta(n))^{s+m}}\\
        &=-a^{-s}\sum_{m \geq 2}\binom{-s}{m}(\frac{b}{a})^mX^{(i)}(s+m)+sa^{-s}\frac{b}{a}X^{(i)}(s+1),
\end{aligned}
\end{equation}
where $X^{(i)}$ represents respectively the functions $G$, $I$ and $J$ for $i=0,1$ and $2$. 

On the other hand, from Proposition 1, $\delta(n)>\delta(1)+(n-1)\frac{1}{\sqrt{5}}\beta=\frac{1}{\sqrt{5}}(1+n\beta)$. Thus, for all $s$ such that $\Re(s)>1$, 

\begin{equation}
  \begin{aligned}
|X^{(i)}(s)|&\leq \sum_{n \geq 1} \frac{1}{|(\delta(n))^s|}\\
        &\leq \sum_{n \geq 1} \frac{1}{(\frac{1}{\sqrt{5}}(1+n\beta))^{\Re(s)}}\\
        &\leq  (\frac{\sqrt{5}}{\beta^2})^{\Re(s)}+\int_{x=1}^{\infty} \frac{1}{(\frac{1}{\sqrt{5}}(1+x\beta))^{\Re(s)}}\\
        &\leq  (\frac{\sqrt{5}}{\beta^2})^{\Re(s)}+\frac{\sqrt{5}}{\beta(\Re(s)-1)} (\frac{\sqrt{5}}{\beta^2})^{\Re(s)-1}.
\end{aligned}
\end{equation}

Consequently, $\sum_{m \geq k}\binom{-s}{m}(\frac{b}{a})^mX^{(i)}(s+m)$ converges for all $s$ such that $\Re(s)>-k$. Combining the fact that $X^{(i)}$ has a meromorphic continuation on $\mathbb{C}$, we prove the meromorphic continuation of $K^{(i)}_{a,b}$ on $\left\{s| \Re(s)>-k\right\}$ for any non-negative number $k$. In particular, for $k=0$, we have the pointwise convergence of $X^{(i)}$ on $\left\{s| \Re(s)>0\right\}$.  
Now, to see the limit at $0$, if $|s|\leq \frac{1}{2}$, then for any positive integer $m$, we have

\begin{equation}
  \begin{aligned}
|\binom{-s}{m}|&=|\frac{|s|\times(|s|+1)\times(|s|+2)...\times(|s|+m-1)}{1\times2\times3...\times m}|\\
        &\leq |s|\prod_{n=1}^m \frac{|s|+n}{n+1}\\
        &\leq |s|.
\end{aligned}
\end{equation}

So that $|-a^{-s}\sum_{m \geq 2}\binom{-s}{m}(\frac{b}{a})^mX^{(i)}(s+m)|\leq |s|a^{-\Re(s)}\sum_{m \geq 2}X^{(i)}(\Re(s)+m)$. With the fact that $\sum_{m \geq 2}X^{(i)}(|s|+m)$ is bounded, we have

\begin{equation}
\lim_{s \to 0} K^{(i)}_{a,b}(s)=\lim_{s \to 0}sa^{-s}\frac{b}{a}X^{(i)}(s+1).
\end{equation}

\end{proof}

\section{Dirichlet series and infinite series}

\subsection{Infinite series involving $(r(n))_{n \in \mathbb{N}^+}$}

Let $(r(n))_{n \in \mathbb{N}^+}$ be the image of the sequence $(d(n))_{n \in \mathbb{N}^+}$ under the map $0\to 0, \;1\to 1,\; 2 \to -1$. From Remark 1 and the description of the sequence A270788, the sequence $(r(n))_{n \in \mathbb{N}^+}$ is the fixed point of the morphism $0\to0,1$; $1\to -1$ and $-1 \to 0,1$.

 Now let us consider the following functions:

$$P(s)=\sum_{n \geq 2}r(n)(\frac{1}{\delta(n-1)^s}-\frac{1}{\delta(n)^s});$$

From Proposition 4 and Proposition 6, $r(n)=1$ if and only if $d(n-1)=0$ and $\delta(n)=\delta(n-1)+\frac{\beta}{\sqrt{5}}$; similarly, $r(n)=-1$ if and only if there exists an $m$ such that $d(m)=0$ and $\tau_0(m)=n-1$; moreover, $\delta(n)=\delta(n-1)+\frac{\beta^2}{\sqrt{5}}$.

 Thus, for all $s$ such that $\Re(s)>1$,

\begin{equation}
  \begin{aligned}
P(s)   &=\sum_{\substack{n \geq 1\\d(n)=1}}(\frac{1}{\delta(n-1)^s}-\frac{1}{\delta(n)^s})-\sum_{\substack{n \geq 1\\d(n)=2}}(\frac{1}{\delta(n-1)^s}-\frac{1}{\delta(n)^s})\\
         &=\sum_{\substack{n \geq 1\\d(n)=0}}(\frac{1}{\delta(n)^s}-\frac{1}{\delta(n+1)^s})-\sum_{\substack{n \geq 1\\d(n)=0}}(\frac{1}{\delta(\tau_0(n))^s}-\frac{1}{(\delta(\tau_0(n)+1))^s})\\
         &=\sum_{\substack{n \geq 1\\d(n)=0}}(\frac{1}{\delta(n)^s}-\frac{1}{(\delta(n)+\frac{\beta}{\sqrt{5}})^s})-\sum_{\substack{n \geq 1\\d(n)=0}}(\frac{1}{(\beta\delta(n))^s}-\frac{1}{(\beta\delta(n)+\frac{\beta^2}{\sqrt{5}})^s})\\
        &=(1-\beta^{-s})\sum_{\substack{n \geq 1\\d(n)=0}}(\frac{1}{\delta(n)^s}-\frac{1}{(\delta(n)+\frac{\beta}{\sqrt{5}})^s}).
\end{aligned}
\end{equation}

This equation yields two consequences. First, from Proposition 7, the infinite sum on the righthand side has a meromorphic continuation on the whole complex plane and $P(0)=0$; secondly, the function $P$ is a derivative on a neighbourhood of $s=0$ and $$P'(s)=\ln(\beta)\beta^{-s}K^{(0)}_{1, \frac{\beta}{\sqrt{5}}}(s)+(1-\beta^{-s})\frac{d}{ds}K^{(0)}_{1, \frac{\beta}{\sqrt{5}}}(s),$$ and thus
\begin{equation}
P'(0)=\ln(\beta)\frac{\beta-1}{\sqrt{5}}.
\end{equation}

Now let us compute an alternative presentation of $P'(0)$. From (14), we can compute further

\begin{equation}
  \begin{aligned}
P(s)   &=(1-\beta^{-s})\sum_{\substack{n \geq 1\\d(n)=0}}(\frac{1}{\delta(n)^s}-\frac{1}{(\delta(n)+\frac{\beta}{\sqrt{5}})^s})\\
         &=(\beta^{-s}-1)\sum_{\substack{n \geq 1\\d(n)=0}}\frac{1}{\delta(n)^s}\sum_{m \geq 1}\binom{-s}{m}(\frac{\beta}{\sqrt{5}\delta(n)})^m\\
         &=(\beta^{-s}-1)\sum_{m \geq 1}\binom{-s}{m}(\frac{\beta}{\sqrt{5}})^m\sum_{\substack{n \geq 1\\d(n)=0}}\frac{1}{\delta(n)^{m+s}}\\
        &=(\beta^{-s}-1)\sum_{m \geq 1}\binom{-s}{m}(\frac{\beta}{\sqrt{5}})^mG(m+s).
\end{aligned}
\end{equation}

 On the other hand, for all $s$ such that $\Re(s)>1$,

\begin{equation}
  \begin{aligned}
P(s)   &=\sum_{\substack{n \geq 1\\d(n)=1}}(\frac{1}{\delta(n-1)^s}-\frac{1}{\delta(n)^s})-\sum_{\substack{n \geq 1\\d(n)=2}}(\frac{1}{\delta(n-1)^s}-\frac{1}{\delta(n)^s})\\
         &=\sum_{\substack{n \geq 1\\d(n)=0}}\frac{1}{\delta(n)^s}-\sum_{\substack{n \geq 1\\d(n)=1}}\frac{1}{\delta(n)^s}-\sum_{\substack{n \geq 1\\d(n)=0}}\frac{1}{\delta(\tau_0(n))^s}+\sum_{\substack{n \geq 2\\d(n)=0}}\frac{1}{\delta(n)^s}\\
         &=(1-\beta^{-1}-\frac{\beta^{-s}}{1-\beta^{-2s}}+\frac{\beta^{-2s}}{1-\beta^{-2s}})\sum_{\substack{n \geq 1\\d(n)=0}}\frac{1}{\delta(n)^s}\\
         &=\frac{1-\beta^{-s}-\beta^{-2s}}{1+\beta^{-s}}G(s).
\end{aligned}
\end{equation}

Substituting (17) into (16), we have

\begin{equation}
  \begin{aligned}
P(s)   &=(\beta^{-s}-1)\sum_{m \geq 1}\binom{-s}{m}(\frac{\beta}{\sqrt{5}})^mG(m+s)\\
         &=(\beta^{-s}-1)\sum_{m \geq 1}\binom{-s}{m}(\frac{\beta}{\sqrt{5}})^m\frac{1+\beta^{-m-s}}{1-\beta^{-m-s}-\beta^{-2m-2s}}P(m+s).
\end{aligned}
\end{equation}

The infinite sum on the righthand side converges uniformly on $\left\{s| \Re(s) >0 \right\}$: thus, $P(1)=\sum_{n \geq 2}r(n)(\frac{1}{\delta(n-1)}-\frac{1}{\delta(n)})$. Moreover, from the fact that $P(0)=0$, to compute $P'(0)$, it is sufficient to compute $\lim_{s \to 0}\frac{P(s)}{s}$. Thus,

\begin{equation}
  \begin{aligned}
\lim_{s \to 0}\frac{P(s)}{s}   &=\lim_{s \to 0}(\beta^{-s}-1)\sum_{m \geq 1}\frac{\binom{-s}{m}}{s}(\frac{\beta}{\sqrt{5}})^m\frac{1+\beta^{-m-s}}{1-\beta^{-m-s}-\beta^{-2m-2s}}P(m+s)\\
         &=\lim_{s \to 0}(1-\beta^{-s})(\frac{\beta}{\sqrt{5}})\frac{1+\beta^{-1-s}}{1-\beta^{-1-s}-\beta^{-2-2s}}P(1+s)\\
         &+\lim_{s \to 0}(\beta^{-s}-1)\sum_{m \geq 2}(-1)^m(\frac{\beta}{\sqrt{5}})^m\frac{1+\beta^{-m-s}}{1-\beta^{-m-s}-\beta^{-2m-2s}}P(m+s)\\
         &=(1+\beta^{-1})(\frac{\beta}{\sqrt{5}})P(1)\lim_{s \to 0}\frac{1-\beta^{-s}}{1-\beta^{-1-s}-\beta^{-2-2s}}\\
         &=(\frac{\beta+1}{\sqrt{5}})P(1)\lim_{s \to 0}\frac{\ln(\beta)\beta^{-s}}{\ln(\beta)\beta^{-1-s}+2\ln(\beta)\beta^{-2-2s}}\\
         &=(\frac{\beta+1}{\sqrt{5}})P(1).
\end{aligned}
\end{equation}

Combining Equation 15, Equation 19 and Remark 1, we have

\begin{proposition}
Letting $(r(n))_{n \in \mathbb{N}^+}$ be the image of the sequence A270788 in OEIS under the map $1 \to 0, 2 \to 1$ and $3 \to -1$, then we have
\begin{equation}
\sum_{n \geq 2}r(n)(\frac{1}{\delta(n-1)}-\frac{1}{\delta(n)})=\frac{\beta-1}{\beta^2}\ln(\beta).
\end{equation}
\end{proposition}

\begin{corollary}
Letting $(d(n))_{n \in \mathbb{N}^+}$ be the sequence defined as above, then we have
\begin{equation}
\sum_{\substack{n \geq 1\\d(n)=1}}\frac{1}{\delta(n-1)}-\frac{1}{\delta(n)}={\beta}^{-1}\ln(\beta).
\end{equation}
\begin{equation}
\sum_{\substack{n \geq 1\\d(n)=2}}\frac{1}{\delta(n-1)}-\frac{1}{\delta(n)}=\beta^{-2}\ln(\beta).
\end{equation}
\end{corollary}

\begin{proof}
First, it is easy to check that the infinite series $\sum_{\substack{n \geq 1\\d(n)=1}}\frac{1}{\delta(n-1)}-\frac{1}{\delta(n)}$ and $\sum_{\substack{n \geq 2\\d(n)=1}}\frac{1}{\delta(n-1)}-\frac{1}{\delta(n)}$ are both well defined. Second, from (14),
$$\sum_{\substack{n \geq 1\\d(n)=2}}\frac{1}{\delta(n-1)}-\frac{1}{\delta(n)}=\beta^{-1}\sum_{\substack{n \geq 1\\d(n)=1}}\frac{1}{\delta(n-1)}-\frac{1}{\delta(n)}.$$
Third, from (20),
$$\sum_{\substack{n \geq 1\\d(n)=1}}\frac{1}{\delta(n-1)}-\frac{1}{\delta(n)}-\sum_{\substack{n \geq 1\\d(n)=2}}\frac{1}{\delta(n-1)}-\frac{1}{\delta(n)}=\frac{\beta-1}{\beta^2}\ln(\beta).$$ Combining the last two equations, we complete the proof.
\end{proof}

Moreover, let us recall the fact
\begin{equation}
\delta(2)=\sum_{n \geq 2}(\delta(n)-\delta(n+1))
\end{equation}
By calculating (23)-(20), we have
\begin{proposition}
Letting $(t(n))_{n \in \mathbb{N}^+}$ be the image of the sequence A270788 in OEIS under the map $1 \to 1, 2 \to 0$ and $3 \to 2$, then we have
\begin{equation}
\sum_{n \geq 2}t(n)(\frac{1}{\delta(n-1)}-\frac{1}{\delta(n)})=\beta^2-\frac{\beta-1}{\beta^2}\ln(\beta).
\end{equation}
\end{proposition}

\subsection{Infinite series involving $(s(n))_{n \in \mathbb{N}^+}$}

Here let us consider another example. Let $(s(n))_{n \in \mathbb{N}^+}$ be a sequence defined in the following way:
$$s(n)=\begin{cases}
-1 \; \text{if}\; d(s)=0 \; \text{or}\; 2\\
2 \; \text{if the Zeckendorf expansion of $n$ admits the string $\overline{1,0}$ as a suffix}\\
1 \; \text{otherwise},
\end{cases}$$ and let us consider the following function:

$$Q(s)=\sum_{n \geq 1}\frac{s(n)}{\delta(n)^s}.$$

From Proposition 6 and Corollary 1, $s(n)=2$ if and only if there exists a $m$, such that $d(m)=0$ and $\tau_0(m)=n$. Thus, for all $s$ such that $\Re(s)>1$,

\begin{equation}
  \begin{aligned}
Q(s)   &=\sum_{\substack{n \geq 1\\d(n)=1}}\frac{1}{\delta(n)^s}-\sum_{\substack{n \geq 0\\d(n)=0}}\frac{1}{\delta(n)^s}-\sum_{\substack{n \geq 1\\d(n)=2}}\frac{1}{\delta(n)^s}+\sum_{\substack{n \geq 1\\d(n)=0}}\frac{1}{\delta(\tau_0(n))^s}\\
         &=\frac{\beta^{-s}}{1-\beta^{-2s}}\sum_{\substack{n \geq 1\\d(n)=0}}\frac{1}{\delta(n)^s}-\sum_{\substack{n \geq 0\\d(n)=0}}\frac{1}{\delta(n)^s}-\frac{\beta^{-2s}}{1-\beta^{-2s}}\sum_{\substack{n \geq 1\\d(n)=0}}\frac{1}{\delta(n)^s}+\beta^{-s}\sum_{\substack{n \geq 1\\d(n)=0}}\frac{1}{\delta(n)^s}\\
         &=\frac{\beta^{-s}+\beta^{-2s}-1}{1+\beta^{-s}}G(s).
\end{aligned}
\end{equation}
Substituting (3) into (25), we have
\begin{equation}
  \begin{aligned}
Q(s)&=\frac{\beta^{-s}+\beta^{-2s}-1}{1+\beta^{-s}}G(s)=\frac{(\beta^{-s}+\beta^{-2s}-1)}{1+\beta^{-s}}((1-\beta^{-s})F(s)-\delta(1)^{-s})\\
&=\frac{(-\beta^{-3s}+2\beta^{-s}-1)}{1+\beta^{-s}}F(s)-\frac{(\beta^{-s}+\beta^{-2s}-1)\delta(1)^{-s}}{1+\beta^{-s}}.
  \end{aligned}
\end{equation}
On the other hand, 
\begin{equation}
Q(s)=\sum_{\substack{n \geq 1\\d(n)=1,\;\text{or}\; 2}}\frac{1}{\delta(n)^s}-\sum_{\substack{n \geq 1\\d(n)=0,\;\text{or}\; 2}}\frac{1}{\delta(n)^s}-(\sum_{\substack{n \geq 1\\d(n)=1,\;\text{or}\; 2}}\frac{1}{\delta(n)^s}-\sum_{\substack{n \geq 1\\d(n)=0}}\frac{1}{\delta(\tau_0(n))^s})+\sum_{\substack{n \geq 0\\d(n)=1}}\frac{1}{\delta(n)^s}.
\end{equation}
From Proposition 4 and Proposition 6,
$$\sum_{\substack{n \geq 1\\d(n)=1,\;\text{or}\; 2}}\frac{1}{\delta(n)^s}=\sum_{n \geq 1}\frac{1}{\delta(\tau_0(n))^s}=\sum_{n \geq 1}\frac{1}{\delta(\tau_0(\tau_0(n)))^s}+\sum_{\substack{n \geq 1\\d(n)=0}}\frac{1}{\delta(\tau_0(n))^s}.$$
$$\sum_{\substack{n \geq 1\\d(n)=0,\;\text{or}\; 2}}\frac{1}{\delta(n)^s}=\sum_{n \geq 1}\frac{1}{\delta(\tau_0(n)+1)^s}+\frac{1}{\delta(1)^s}$$
$$\sum_{\substack{n \geq 1\\d(n)=1}}\frac{1}{\delta(n)^s}=\sum_{\substack{n \geq 1\\d(n)=0,\;\text{or}\; 2}}\frac{1}{\delta(\tau_0(n))^s}=\sum_{n \geq 1}\frac{1}{\delta(\tau_0(\tau_0(n)+1))^s}+\frac{1}{\delta(2)^s}$$
Thus,
\begin{equation}
  \begin{aligned}
Q(s)  &=\sum_{n \geq 1}\frac{1}{\delta(\tau_0(n))^s}-(\sum_{n \geq 1}\frac{1}{\delta(\tau_0(n)+1)^s}+\frac{1}{\delta(1)^s})-\sum_{n \geq 1}\frac{1}{\delta(\tau_0(\tau_0(n)))^s}+\sum_{n \geq 1}\frac{1}{\delta(\tau_0(\tau_0(n)+1))^s}+\frac{1}{\delta(2)^s}\\
         &=\sum_{n \geq 1}(\frac{1}{(\beta\delta(n))^s}-\frac{1}{(\beta\delta(n)+\frac{\beta^2}{\sqrt{5}})^s})-\sum_{n \geq 1}(\frac{1}{(\beta^2\delta(n))^s}-\frac{1}{(\beta^2\delta(n)+\frac{\beta^3}{\sqrt{5}})^s})-\frac{1}{\delta(1)^s}+\frac{1}{\delta(2)^s}\\
         &=(1-\beta^{-s})\sum_{m \geq 1}(\frac{1}{(\beta\delta(n))^s}-\frac{1}{(\beta\delta(n)+\frac{\beta^2}{\sqrt{5}})^s}))-\frac{1}{\delta(1)^s}+\frac{1}{\delta(2)^s}.
\end{aligned}
\end{equation}

From Proposition 7, for all $s$ such that $\Re(s)>1$,

\begin{equation}
  \begin{aligned}
Q(s)  &=(1-\beta^{-s})\sum_{m \geq 1}(\frac{1}{(\beta\delta(n))^s}-\frac{1}{(\beta\delta(n)+\frac{\beta^2}{\sqrt{5}})^s}))-\frac{1}{\delta(1)^s}+\frac{1}{\delta(2)^s}\\
         &=(\beta^{-2s}-1)\sum_{m \geq 1}\binom{-s}{m}(\frac{\beta}{\sqrt{5}})^mF(s+m)-\frac{1}{\delta(1)^s}+\frac{1}{\delta(2)^s}.
\end{aligned}
\end{equation}
Substituting (26) into (29), we have
\begin{equation}
  \begin{aligned}
Q(s)&=(\beta^{-2s}-1)\sum_{m \geq 1}\binom{-s}{m}(\frac{\beta}{\sqrt{5}})^m\frac{1+\beta^{-m-s}}{-\beta^{-3(m+s)}+2\beta^{-(m+s)}-1}(Q(s+m)+\frac{(\beta^{-m-s}+\beta^{-2m-2s}-1)\delta(1)^{-m-s}}{1+\beta^{-m-s}})\\
&-\frac{1}{\delta(1)^s}+\frac{1}{\delta(2)^s}\\
  \end{aligned}
\end{equation}
From Equation 28 and Proposition 7, we can prove that the function $Q$  has a meromorphic continuation on the whole complex plane. Moreover, we have $Q(0)=0$ and $$Q'(s)=\ln(\beta)\beta^{-s}K^{(3)}_{\beta, \frac{\beta^2}{\sqrt{5}}}(s)+(1-\beta^{-s})\frac{d}{ds}K^{(3)}_{\beta, \frac{\beta^2}{\sqrt{5}}}(s)-\ln(\beta),$$ thus
\begin{equation}
Q'(0)=\ln(\beta)\frac{\beta}{\sqrt{5}}-\ln(\beta).
\end{equation}

 From (30), the infinite sum on the righthand side converges uniformly on $\left\{s| \Re(s) >0 \right\}$, thus $Q(1)=\sum_{n \geq 1}\frac{s(n)}{\delta(n)}$. To compute $Q'(0)$, it is sufficient to compute $\lim_{s \to 0}\frac{Q(s)}{s}$. Thus,

\begin{equation}
  \begin{aligned}
\lim_{s \to 0}\frac{Q(s)}{s}   &=\lim_{s \to 0}(\beta^{-2s}-1)\sum_{m \geq 1}\binom{-s}{m}(\frac{\beta}{\sqrt{5}})^m(\frac{1+\beta^{-m-s}}{-\beta^{-3(m+s)}+2\beta^{-(m+s)}-1}Q(s+m)+\frac{1}{1-\beta^{-m-s}}\delta(1)^{-m-s})\\
         &+\lim_{s \to 0}\frac{-\frac{1}{\delta(1)^s}+\frac{1}{\delta(2)^s}}{s}\\
         &=\lim_{s \to 0}(1-\beta^{-2s})(\frac{\beta}{\sqrt{5}})(\frac{1+\beta^{-m-s}}{-\beta^{-3(1+s)}+2\beta^{-(1+s)}-1}Q(s+1)+\frac{1}{1-\beta^{-1-s}}\delta(1)^{-1-s})\\
         &-\lim_{s \to 0}\frac{\frac{1}{\delta(1)^s}-1}{s}+\lim_{s \to 0}\frac{\frac{1}{\delta(2)^s}-1}{s}\\
         &+\lim_{s \to 0}(\beta^{-2s}-1)\sum_{m \geq 2}\binom{-s}{m}(\frac{1+\beta^{-m-s}}{-\beta^{-3(m+s)}+2\beta^{-(m+s)}-1}Q(s+m)+\frac{1}{1-\beta^{-m-s}}\delta(1)^{-m-s})\\
         &=(\frac{\beta}{\sqrt{5}})Q(1)\lim_{s \to 0}\frac{1-\beta^{-2s}}{-\beta^{-3(1+s)}+2\beta^{-(1+s)}-1}(1+\beta^{-1-s})+\ln(\delta(1))-\ln(\delta(2))\\
         &=(\frac{\beta}{\sqrt{5}})Q(1)(1+\beta^{-1})\lim_{s \to 0}\frac{2\ln(\beta)\beta^{-2s}}{3\ln(\beta)\beta^{-3-3s}-2\ln(\beta)\beta^{-1-s}}-\ln(\beta)\\
         &=(\frac{\beta}{\sqrt{5}})Q(1)\frac{2+2\beta^{-1}}{3\beta^{-3}-2\beta^{-1}}-\ln(\beta).
\end{aligned}
\end{equation}
Combining (31) and (32), we have
\begin{proposition}
Letting $(s(n))_{n \in \mathbb{N}^+}$ be the sequence defined as above, then
\begin{equation}
\sum_{n \geq 2}\frac{s(n)}{\delta(n)}=(\frac{3}{2}\beta^{-4}-\beta^{-2})\ln(\beta).
\end{equation}
\end{proposition}
To obtain the Theorem 2, we only need to apply the following proposition:
\begin{proposition}
For any positive integer $n$, $\delta(n)=n-\frac{\{\beta n\}-1+f(n)}{\sqrt{5}},$ where $\{x\}$ is the fractional part of $x$.
\end{proposition}
\begin{proof}
Letting $n$ be a positive integer, then $n=\delta(n)-\delta'(n)$ and $\tau_0(n)=\beta\delta(n)+\beta^{-1}\delta'(n)$. Consequently, $\beta n-\tau_0(n)=-(\beta+\beta^{-1})\delta'(n)=-\sqrt{5}\delta'(n)$. Moreover, from Proposition 2, $|\delta'(n)|<\frac{1}{\sqrt{5}\beta}$, we have $|\sqrt{5}\delta'(n)|<1$. Thus, $\{\beta n\}=1-\sqrt{5}\delta'(n)$ if $\delta'(n)>0$ and $\{\beta n\}=-\sqrt{5}\delta'(n)$ if $\delta'(n)<0$. Consequently, $\delta(n)=n-\frac{\{\beta n\}-1}{\sqrt{5}}$ if $f(n)=0$ and $\delta(n)=n-\frac{\{\beta n\}}{\sqrt{5}}$ if $f(n)=1$.
\end{proof}

\bibliographystyle{alpha}
\bibliography{citations_V4}

\end{document}